\documentclass{amsart}

\newtheorem{theorem}{Theorem}[section]
\newtheorem{lemma}[theorem]{Lemma}

\newtheorem{claim}[theorem]{Claim}

\theoremstyle{definition}

\theoremstyle{remark}
\newtheorem{remark}[theorem]{Remark}

\numberwithin{equation}{section}

\begin{document}

\title{ A gap theorem of four-dimensional  gradient shrinking solitons}

\author{Zhuhong Zhang}
\address{Department of Mathematics, South China Normal Univeristy, Guangzhou, P. R. China 510275}
\email{juhoncheung@sina.com}






\begin{abstract}
In this paper, we will prove a gap theorem for four-dimensional gradient shrinking  soliton. More precisely,  we will show that any complete four-dimensional gradient shrinking soliton with nonnegative and bounded Ricci curvature, satisfying a pinched Weyl curvature,  either is flat, or $\lambda_1 + \lambda_2\ge c_0 R>0$ everywhere for some $c_0\approx 0.29167$, where $\{\lambda_i\}$ are the two least eigenvalues of Ricci curvature. Furthermore, we will show that $\lambda_1 + \lambda_2\ge \frac 13R>0$ under a better pinched Weyl tensor assumption. We point out that the lower bound $\frac 13R$ is sharp.
\end{abstract}

\maketitle



\section{Introduction}

A Riemannian manifold $(M, g)$, couple with  a smooth function $f$, is called gradient Ricci soliton, if there is a constant $\rho$, such that 
$$R_{ij}+ \nabla_i\nabla_j f= \rho g_{ij}.$$
The soliton is called shrinking, steady, or expanding, if $\rho>0$, $\rho=0$, or $\rho<0$, respectively.
Gradient shrinking solitons (GSS for short) play an important role in  the Ricci flow, as they correspond to self-similar solutions, and often arise naturally as limits of dilations of Type I singularities of Ricci flow. They are also generalizations of Einstein metrics. Thus it is a central issue to understand and  classify GSS.

The GSS are complete classified in dimension 2 (see \cite{Ha95}) and 3 (see \cite{Ivey, P2, NW, CCZ}), and in dimension $n\ge 4$ with vanishing Weyl tensor (see \cite{NW, PW, ZH09}). In recent years, there are some other attention to the classification of complete GSS (see \cite{Na, ChenW, Catino, WWW, LNW} ).

For a better understanding and ultimately for the classifications of GSS in higher dimension, one tries to obtain some curvature estimates and other geometric structures on GSS. 
In particular, on a complete non-compact GSS,  Chen \cite{Chen} showed that it will have nonnegative scalar curvature. In addition, Cao-Zhu\cite{CZ08} showed that it has infinite volume (or see  \cite{Cao09} Theorem 3.1). While Cao-Zhou\cite{CaoZhou} obtained a rather precise estimate on asymptotic behavior of the potential function $f$, and showed that it must have at most Euclidean volume growth. 

If the GSS further satisfies some curvature assumptions, then we can get some more precise characteristics. For example, Carrillo-Ni \cite{CN} showed that any GSS with nonnegative Ricci curvature must have zero asymptotic volume ratio, and Munteanu-Wang \cite{MW15} proved that GSS with nonnegative sectional curvature and positive Ricci curvature must be compact. In \cite{MW14}, Munteanu-Wang obtained some curvature estimates on four-dimensional GSS with bounded scalar curvature.
In this paper, we obtain a gap theorem of four-dimensional GSS with pinched curvature. 

Let  $(M^n, g)$ be a complete Riemannian manifold. Denote by $Ric=\{R_{ij}\}$ and $R$ are the Ricci tensor and scalar curvature respectively. It is well known that the Riemannian  curvature tensor $Rm=\{R_{ijkl}\}$ can be decomposed  into the orthogonal components :  
$$Rm=W\oplus \frac 2{n-2} \mathring{Ric} \wedge g \oplus \frac{R}{n(n-1)} g\wedge g,$$
where $W=\{W_{ijkl}\}$ is the Weyl tensor, and $\mathring{Ric}=\{R_{ij}-\frac Rn g_{ij}\}$ is the traceless Ricci curvature. Now we can state 
our main theorem.
\begin{theorem}
Let $(M^4, g)$ be a  complete four-dimensional GSS with bounded and nonnegative Ricci curvature $0\le Ric \le C$, satisfying
$$|W| \le \gamma \Big| |\mathring{Ric} | - \frac 1{2\sqrt{3}}R \Big| \eqno{(*)}$$
for some constant $\gamma<1+\sqrt{3}$.
Then either the soliton is flat, or $$\lambda_1+\lambda_2 \ge c_0 R>0,$$
for some positive constant $c_0=\frac   { (1+2\sqrt{3})   -  \sqrt{5+4\sqrt{3}}   }  {2\sqrt{3}} \approx 0.29167$,
where $\lambda_1$ and $\lambda_2$ are the least two eigenvalues of the Ricci curvature.
\end{theorem}

\begin{remark}
In view of the round cylinder $\mathbb{S}^2\times \mathbb{R}^2$ with constant scalar curvature,  the pinched constant $\gamma$ in $(*)$ is necessary. 
Indeed,  $\mathbb{S}^2\times \mathbb{R}^2$ is a non-flat GSS with Ricci curvature $0\le Ric \le \frac 12 R$. Furthermore, 
$| \mathring{Ric} |=\frac 12 R$, and the Weyl tensor satisfies 
$$|W| = \frac 1 {\sqrt{3} } R =(1+\sqrt{3}) \Big| | \mathring{Ric} | - \frac 1{2\sqrt{3}}R \Big|. $$
But the least two eigenvalues of the Ricci curvature $\lambda_1+\lambda_2 \equiv 0$ everywhere.
\end{remark}

Follow by a similar argument, we can show a better result under a better pinched condition as follow.
\begin{theorem}
Let $(M^4, g)$ be a  complete four-dimensional GSS with bounded and nonnegative Ricci curvature $0\le Ric \le C$, satisfying
$$|W| \le \gamma \Big| |\mathring{Ric} | - \frac 1{2\sqrt{3}}R \Big| \eqno{(**)}$$
for some constant $\gamma  \le  \frac{1+\sqrt{3}}{\sqrt{3}}$.
Then either the soliton is flat, or $$\lambda_1+\lambda_2 \ge \frac 13 R>0,$$
where $\lambda_1$ and $\lambda_2$ are the least two eigenvalues of the Ricci curvature.
\end{theorem}

\begin{remark}
Our conclusion $\lambda_1+\lambda_2 \ge \frac 13 R>0$ is sharp due to the example of round cylinder $\mathbb{S}^3\times \mathbb{R}$. Since $\mathbb{S}^3\times \mathbb{R}$ is also a non-flat GSS with Ricci curvature $0 \le Ric \le \frac 13 R$, and $|\mathring{Ric} |=\frac 1{2\sqrt{3}}R$, $|W|=0$. These facts imply that  pinched condition $(*)$ holds. But the least two eigenvalues of the Ricci curvature $\lambda_1+\lambda_2 \equiv \frac 13 R$ everywhere.
\end{remark}

 {\bf Acknowledgements} The author was partially supported by NSFC 11301191.

\section{Preliminaries}

Let $(M^4, g_{ij})$ be a four-dimensional Riemannian manifold. We can decompose the 
Riemannian curvature tensor $\{R_{ijkl}\}$  as follows:  
\begin{align*}
R_{ijkl}=&  W_{ijkl}+\frac{1}{2}(R_{ik}g_{jl}+R_{jl}g_{ik}-R_{il}g_{jk}-R_{jk}g_{il})  \\ 
              &    -\frac{1}{6}R(g_{ik}g_{jl}-g_{il}g_{jk}).
\end{align*}

In this section, we suppose $(M^4, g_{ij})$ is a complete Riemannian manifold with bounded curvature. Now we consider the Ricci flow equation
$$  \left\{
       \begin{array}{lll}
       & \frac{\partial g_{ij}(x, t)}{\partial t}=-2R_{ij}(x,t ), &x\in M^4, t>0, \\[2mm]
       & g_{ij}(x, 0)=g_{ij}(x) , & x\in M^4.
       \end{array}
    \right.
$$
Since the curvature is bounded at the initial metric, it is well known \cite{Shi} that there exist a complete solution $g(t)$ of the Ricci flow on 
a time interval $[0, T)$ with bounded curvature for each $t$.  Furthermore,  the Ricci curvature tensor $\{R_{ij}\}$ and the scalar curvature $R$
evolve by the (PDE) system under the moving frame  (cf. Hamilton \cite{Ha86}):
$$
\begin{aligned}
\frac{\partial}{\partial t}R_{ij} &= \triangle  R_{ij}  + 2\sum\limits_{k,l}R_{ikjl}R_{kl}  \\
 \frac{\partial}{\partial t}R &=\triangle  R + 2|Ric|^2,  \\ 
\end{aligned} \eqno(\rm{PDE})
$$

By using Hamilton's maximum principle for tensor\cite{Ha86}, a tensor evolves by a nonlinear heat equation may be controlled by a corresponding (ODE) system.  While the (ODE) system corresponding to the above (PDE) is the following 
$$
\begin{aligned}
\frac{d}{d t}R_{ij} &=  2\sum\limits_{k,l}R_{ikjl}R_{kl}  \\
 \frac{d}{d t}R &= 2|Ric|^2.  \\ 
\end{aligned} \eqno(\rm{ODE})
$$

By a direct computation, we have the following lemma.

\begin{lemma} 
Let $b=(\lambda_3 + \lambda_4) - (\lambda_1 + \lambda_2)$, where $\{ \lambda_i\}$ are the eigenvalues of the Ricci tensor with $\lambda_1\le \lambda_2 \le \lambda_3 \le \lambda_4$. Then under the (ODE) system, we have
$$\frac 12 \frac d{dt}b \le 2b(\frac R3 + W_{1212} )  + (\lambda_1^2 + \lambda_2^2) - (\lambda_3^2 + \lambda_4^2 ) .$$
\end{lemma}    

\begin{proof}
Indeed, since $W_{ijij}=W_{klkl}$ and $\sum\limits_i W_{ijiij}=0$
for any orthonormal four-frame $\{e_i, e_j, e_k, e_l\}$, we have
\begin{align*}
\frac 12 \frac{d}{d t} \Big( \lambda_1 + \lambda_2 \Big)
\ge& \sum\limits_{k=2,3,4} \lambda_k \Big( W_{1k1k}+\frac{\lambda_1+\lambda_k}2-\frac R6 \Big)  \\
     &    + \sum\limits_{l=1,3,4} \lambda_l \Big( W_{2l2l}+\frac{\lambda_2+\lambda_l}2-\frac R6 \Big)  \\
=& (\lambda_1+ \lambda_2) \Big( W_{1212}+\frac{\lambda_1+\lambda_2}2-\frac R6 \Big)  \\ 
  & +  \lambda_3 \Big(-W_{1212}+\frac{\lambda_1+\lambda_2}2 + \lambda_3 -\frac R3 \Big) \\
   &+  \lambda_4 \Big(-W_{1212}+\frac{\lambda_1+\lambda_2}2 + \lambda_4 -\frac R3 \Big)  \\
=& (W_{1212} + \frac R3) (\lambda_1+ \lambda_2-\lambda_3- \lambda_4)    + \lambda_3^2 + \lambda_4^2. 
\end{align*} 

Similarly, we have
$$\frac 12 \frac{d}{d t} \Big( \lambda_3 + \lambda_4 \Big) \le 
(W_{3434}  + \frac R3) (\lambda_3+ \lambda_4-\lambda_1- \lambda_2)    + \lambda_1^2 + \lambda_2^2 .$$

Subtract the above two inequalities, we obtain our assertion.
\end{proof}

\section{ A key pinched estimate } 

In this section, we will give a pinched estimate, which implies that the curvature $b$ described in Lemma 2.1 can become better and better under the Ricci flow.

\begin{lemma} 
Suppose we have a solution of Ricci flow $g(t)_{t\in [0, T]}$ on a four-manifold with uniformly bounded and nonnegative Ricci curvature, and satisfying the pinched condition $(*)$ 
for all $t \in [0, T]$.

If $R\ge r_0$ and $b \le \eta_0 R \le R$
holds for some positive constant $r_0>0$ and $\eta_0 > \tilde{c}$ at time $t=0$, where $\tilde{c} = \frac   {\sqrt{5+4\sqrt{3}}  -  (1+\sqrt{3})  }  {\sqrt{3}} \approx 0.41666 > \frac 13$ . Then there exist a positive constant $\delta =\delta(r_0, \eta_0, \gamma) \in (0, 1]$, such that 
$$b \le (\eta_0 -\delta t )R$$
holds for all $t \in [0, T']$, where $T'=\min\{ T, \frac{ \eta_0  - \tilde{c} }2  \}$.
\end{lemma}

\begin{proof}
Note that both the Ricci curvature tensor and the Weyl tensor are uniformly bounded, hence $g(t)$ has uniformly bounded curvature.

Consider the set $\Omega (t)_{t\in [0, T']}$ of matrices defined by the inequalities
$$ \Omega (t):  \left \{
       \begin{array}{lll}
       & R\ge r_0, \\[2mm]
       & b \le (\eta_0 -\delta t )R.\\
       \end{array}
    \right.
$$
The constant $\delta\in (0, 1]$ will be choosed later.

It is easy to see that $\Omega (t)$ is closed, convex and $O(n)$-invariant. By the assumptions at $t=0$ and the Hamilton's maximum principle for tensor, we only need to show the set $\Omega(t)$ is preserved by the (ODE) system. Indeed, we only need to look at points on the boundary of the set.

From the (ODE) system, we have
$$ \frac{d}{d t}R = 2|Ric|^2\ge 0,$$
which implies that $R\ge r_0$ for all $t\ge 0$. Thus the first inequality is preserved. To prove the second inequality, we only need to show that
$$\frac 12 b'\le (\eta_0 - \delta t ) \frac 12R' - \frac{ \delta}2 R =\eta \cdot \frac 12R' - \frac{ \delta}2 R, $$
where $b = (\eta_0-\delta t )R= \eta R $. 

By Lemma 2.1 and the (ODE) system,  it is suffice to show that
$$2b(\frac R3 + W_{1212} )  + (\lambda_1^2 + \lambda_2^2) - (\lambda_3^2 + \lambda_4^2 )\le \eta \sum\limits_i \lambda_i^2 - \frac{ \delta}2 R.  $$

It is equivalent to show that
$$I = (1+\eta ) (\lambda_3^2 + \lambda_4^2 ) - (1-\eta)(\lambda_1^2 + \lambda_2^2)  - 2\eta R (\frac R3 + W_{1212} )   \ge  \frac{ \delta}2 R.  \eqno(3.1)$$

Now $b = \eta R $, thus $\lambda_3+ \lambda_4 = \frac {1+\eta}2 R$ and $\lambda_1+ \lambda_2 = \frac {1-\eta}2 R$.
Denote by $x=\frac {\lambda_2 - \lambda_1 }2$ and $y=\frac {\lambda_4 - \lambda_3 }2$, which  satisfies
$$0\le x\le \frac {1-\eta}4 R,\quad y\ge 0, \quad x+y \le \frac {\eta}2 R. $$ 
And then
\begin{align*}
&\lambda_1=\frac {1-\eta}4 R-x, &\lambda_2=\frac {1-\eta}4 R+x,  \\
&\lambda_3=\frac {1+\eta}4 R-y, &\lambda_4=\frac {1+\eta}4 R+y. 
\end{align*}
Meanwhile, by a direct computation, we have
$$ W_{12}^2 \le \frac 23 \sum W_{1k}^2 \le \frac 23 \cdot \frac 18 |W|^2  \le \frac 1{12}\gamma^2  \Big( |\mathring{Ric}| - \frac 1{2\sqrt{3}}R \Big)^2. $$

In the following, we divide the argument into two cases.

{\bf Case 1: $ |\mathring{Ric} | \ge  \frac R{2\sqrt{3}} $. } In this case, 
\begin{align*}
|\mathring{Ric} |^2 =& \sum_i (\frac R4 - \lambda_i)^2  =\lambda_1^2 + \lambda_2^2 +  \lambda_3^2 + \lambda_4^2- \frac 14 R^2\\
=&   (\frac {1-\eta}2 )^2 R^2 - 2 \lambda_1 \lambda_2    + 2 (\frac {1+\eta}4 )^2 R^2 + 2 y^2   - \frac 14 R^2\\
=&    \frac 18 \cdot ( 3\eta^2 -2\eta +1) R^2 + 2 y^2 - 2 \lambda_1 \lambda_2 . \\
\end{align*}

Thus
$$W_{1212} \le  \frac {\gamma}{2\sqrt{3}} \cdot  \Big(  \sqrt{         \frac 18 \cdot ( 3\eta^2 -2\eta +1)R^2  + 2 y^2 - 2 \lambda_1 \lambda_2         }  - \frac R{2\sqrt{3}} \Big).  $$

So $I$ defined in (3.1) can be calculated as follow :
\begin{align*}
I=&  (1+ \eta) \Big[ 2 (\frac {1+\eta}4 R)^2 + 2 y^2 \Big] - (1-\eta) \Big[ (\frac {1-\eta}2 R)^2 - 2 \lambda_1 \lambda_2 \Big]  \\
   & - \frac 23 \eta R^2 - 2\eta R W_{1212}  \\
\ge&   \frac 1{24} ( -3 + 11 \eta - 9\eta^2 + 9 \eta^3  )   R^2 +  2(1+ \eta) y^2 +2 (1-\eta)  \lambda_1 \lambda_2  \\
     &  - 2\eta R \cdot  \frac {\gamma}{2\sqrt{3}} \cdot  \Big(  \sqrt{         \frac 18 \cdot ( 3\eta^2 -2\eta +1)R^2  + 2 y^2 - 2 \lambda_1 \lambda_2         }  - \frac R{2\sqrt{3}} \Big)  \\
 \ge&  \frac 1{24} ( -3 + 11 \eta - 9\eta^2 + 9 \eta^3  )  R^2 +  2(1+ \eta) y^2    \\
       & -  \frac {\gamma \eta R }{\sqrt{3}} \cdot  \Big[  \sqrt{         \frac 18 \cdot ( 3\eta^2 -2\eta +1) R^2 + 2 y^2          }  - \frac R{2\sqrt{3}}   \Big]  .  \\
\end{align*}

Denote by  $t=\sqrt{         \frac 18 \cdot ( 3\eta^2 -2\eta +1) R^2 + 2 y^2          } $. Since $y \ge 0$, we then have 
$$t\ge \frac R{2\sqrt{2}}  \sqrt{ 3\eta^2 -2\eta +1} .$$

And the above inequality becomes 
\begin{align*}
I \ge&  \frac 1{24} ( -3 + 11 \eta - 9\eta^2 + 9 \eta^3  )  R^2 \\
        & +  (1+ \eta) \Big[ t^2 -    \frac 18 \cdot ( 3\eta^2 -2\eta +1) R^2     \Big]     
           -  \frac {\gamma \eta R }{\sqrt{3}} \cdot  \Big( t - \frac R{2\sqrt{3}} \Big)\\
=& (1+ \eta) t^2  - \frac {\gamma \eta R}{\sqrt{3}} \cdot t   \\
  &  +  \frac 1{24} ( -3 + 11 \eta - 9\eta^2 + 9 \eta^3  ) R^2 -    \frac 18 \cdot (1+ \eta) \cdot ( 3\eta^2 -2\eta +1) R^2  +   \frac {\gamma \eta }{6}R^2  \\
\end{align*}

The RHS is a quadratic function of $t$, and increase respect to $t$. Indeed, we only need to show that $2(1+ \eta) t  - \frac {\gamma \eta R}{\sqrt{3}} >0$ holds for all $t\ge  \frac R{2\sqrt{2}}  \sqrt{ 3\eta^2 -2\eta +1}$.
It is easy to see that $ 12 (  \sqrt{3} - 1 ) > ( 1+\sqrt{3})^2> \gamma^2$, and then we have 
$$\frac{  (1+\eta) \sqrt{ 3\eta^2 -2\eta +1}  }{\eta} = ( \frac1{\sqrt{\eta}}+\sqrt{\eta} ) \sqrt{ 3\eta -2 +\frac1{\eta}} \ge 2\sqrt{ 2\sqrt{3} -2 } > \gamma.$$
Thus $2(1+ \eta) t\ge \frac R{\sqrt{2}}\cdot (1+ \eta)    \sqrt{ 3\eta^2 -2\eta +1}> \frac R{\sqrt{2}}\cdot \eta\gamma>\frac {\gamma \eta R}{\sqrt{3}} $.

The above monotonic property respect to $t$ implies that the RHS achieves the minimum value if and only if $t$ takes the minimum value, which is equivalent with $y=0$. Thus
\begin{align*}
I \ge& \frac 1{24} ( -3 + 11 \eta - 9\eta^2 + 9 \eta^3  )R^2   -  \frac {\gamma\eta R^2}{\sqrt{3}} \cdot  \Big(  \frac 1{2\sqrt{2}}  \sqrt{ 3\eta^2 -2\eta +1}  - \frac 1{2\sqrt{3}} \Big)\\
=& \frac {R^2}{24} \Big[ 9 (\eta - \frac 13)  \Big(   (\eta - \frac 13)^2 + \frac 89 \Big)    - 4 \gamma \eta \cdot \Big(   \sqrt{ 1 + \frac 92 (\eta - \frac 13)^2 }   -1   \Big)  \Big]\\
=& \frac {  3(\eta - \frac 13)R^2 }{8} \Big[   (\eta - \frac 13)^2 + \frac 89     -  2\gamma \cdot   \frac { \eta (\eta - \frac 13)}{  \sqrt{ 1 + \frac 92 (\eta - \frac 13)^2 } +1 }  \Big]\\
=& \frac {  3(\eta - \frac 13)R^2 }{8} \Big[ II +  2( 1+ \sqrt{3}- \gamma )\cdot   \frac { \eta (\eta - \frac 13)}{  \sqrt{ 1 + \frac 92 (\eta - \frac 13)^2 } +1 }  \Big],\\
\end{align*}
where
\begin{align*}
II =&  (\eta - \frac 13)^2 + \frac 89     -  2 ( 1+ \sqrt{3} ) \cdot   \frac { \eta (\eta - \frac 13)}{  \sqrt{ 1 + \frac 92 (\eta - \frac 13)^2 } +1 } \\
=&  \eta^2 - \frac 23 \eta +1     -  2 \eta (\eta - \frac 13)\\
  & -  2 \eta (\eta - \frac 13) \cdot  \Bigg[ \frac { 1+ \sqrt{3} }{  \sqrt{ 1 + \frac 92 (\eta - \frac 13)^2 } +1 } -1  \Bigg]\\
=& (1- \eta)(1+\eta)
   -  2 \eta (\eta - \frac 13) \cdot \frac { \sqrt{3} - \sqrt { 1+ \frac 92 (\eta - \frac 13)^2}  }{  \sqrt{ 1 + \frac 92 (\eta - \frac 13)^2 } +1 } \\
=& (1- \eta)(1+\eta)
   -  2 \eta (\eta - \frac 13) \cdot \frac 1{  \sqrt{ 1 + \frac 92 (\eta - \frac 13)^2 } +1 } \cdot  \frac { \frac 32 (1- \eta)(1+3\eta) }{ \sqrt{3} + \sqrt { 1+ \frac 92 (\eta - \frac 13)^2}  }   .  \\
\end{align*}

Note that $\eta =\eta_0 - \delta t \in [ \frac {\eta_0 + \tilde{c}}2, 1] \subset ( \frac 13, 1]$, thus
\begin{align*}
II \ge& (1+\eta) (1- \eta) -  2\eta \cdot \frac 23 \cdot \frac 1{  1 +1 } \cdot  \frac { \frac 32 (1- \eta)\cdot 4 }{ \sqrt{3} +1 }  \\
    =&  (1-\eta)\Big( 1+ \eta - \frac 4{1+\sqrt{3}} \eta \Big) \ge 0,
\end{align*}
and then
\begin{align*}
I \ge& \frac {  3(\eta - \frac 13)R^2 }{8} \cdot 2( 1+ \sqrt{3}- \gamma )\cdot   \frac { \eta (\eta - \frac 13)}{  \sqrt{3 } +1 }  \\
\ge& C_1(\eta_0, \gamma)R^2 \ge C_2(c_0, \eta_0, \gamma)R
\end{align*}
for some positive constant $C_2(r_0, \eta_0, \gamma)>0$.

{\bf Case 2: $ |\mathring{Ric} | <  \frac R{2\sqrt{3}} $. }
In this case, $$|\mathring{Ric} |^2 = 2( \frac{1-\eta}4 )^2 + 2( \frac{1-\eta}4 )^2 - \frac 14 + 2y^2 + 2 x^2=\frac 14 \eta^2 R^2 + 2y^2 + 2 x^2 ,$$
which implies that $\eta =\eta_0 - \delta t \in [ \frac {\eta_0 + \tilde{c}}2,  \frac 1{\sqrt{3}}) \subset ( \tilde{c}, \frac 1{\sqrt{3}})\subset ( \frac 13, \frac 1{\sqrt{3}})$, and
$$-W_{1212} \ge  \frac {\gamma}{2\sqrt{3}} \cdot  \Big(  \sqrt{      \frac 14 \eta^2 R^2 + 2y^2 + 2 x^2   }  - \frac R{2\sqrt{3}} \Big) . $$ 

By a direct computation, we have
\begin{align*}
I \ge& \frac {(1+ \eta)^3}8 R^2 -\frac {(1- \eta)^3}8 R^2 -\frac 23\eta R^2  +  2(1+ \eta) y^2  -2 (1-\eta) x^2   \\
       & + 2\eta R \cdot  \frac {\gamma}{2\sqrt{3}} \cdot  \Big(  \sqrt{ \frac 14 \eta^2 R^2 + 2y^2 + 2 x^2    } - \frac R{2\sqrt{3}} \Big)\\
  =& \frac {\eta}{12} (3\eta^2  + 1) R^2  +  2(1+ \eta) y^2  -2 (1-\eta) x^2   \\
       & +  \frac {\gamma \eta R}{\sqrt{3}} \cdot  \Big(  \sqrt{ \frac 14 \eta^2 R^2 + 2y^2 + 2 x^2    } - \frac R{2\sqrt{3}} \Big)\\
  \ge& \frac {\eta}{12} (3\eta^2 +1)  R^2  -2 (1-\eta) x^2  +  \frac {\gamma \eta R}{\sqrt{3}} \cdot  \Big(  \sqrt{ \frac 14 \eta^2 R^2 +  2 x^2    } - \frac R{2\sqrt{3}} \Big) .  \\
\end{align*}

Denote by $\tau=\sqrt{ \frac 14 \eta^2 R^2 +  2 x^2    } $. Since $x \in [0, \frac {1-\eta}4 R]$, we then have $$\tau \in \Big[ \frac {\eta}2 R, \  \frac R{2\sqrt{2}}  \sqrt{ 3\eta^2 -2\eta +1 }  \Big], \ and$$
$$ I \ge \frac {\eta}{12} (3\eta^2 + 1)  R^2  - (1-\eta) ( \tau^2 - \frac 14 \eta^2 R^2 )  +  \frac {\gamma \eta R}{\sqrt{3}} \cdot  \Big( \tau  - \frac R{2\sqrt{3}} \Big). $$

Similarly, the RHS is a quadratic function of $\tau$, and it can achieve the minimum value where $t$ take endpoint values, which is equivalent with $x=0$ or $x=\frac {1-\eta}4 R$.

If $x=\frac {1-\eta}4R$. Then
\begin{align*}
RHS \ge& \frac {\eta R^2}{12} (3\eta^2  + 1)   - \frac 18 (1-\eta)^3 R^2  \\
              &+  \frac {\gamma\eta R^2}{\sqrt{3}} \cdot  \Big( \frac 1{2\sqrt{2}} \sqrt{ 3\eta^2 -  2\eta +1    } - \frac 1{2\sqrt{3}} \Big) . 
\end{align*}

Since $\eta>\frac 13$, we have $\frac 1{2\sqrt{2}} \sqrt{ 3\eta^2 -  2\eta +1    } > \frac 1{2\sqrt{2}} \sqrt{ \frac 23    } =\frac 1{2\sqrt{3}} $. Thus
\begin{align*}
RHS>& \frac {\eta R^2}{12} (3\eta^2  + 1)   - \frac 18 (1-\eta)^3 R^2  \\
=& \frac {R^2}{24} \Big[  2\eta\cdot  (   3\eta^2 + 1 )   -  3 (1-\eta)^3    \Big] \\
=& \frac {R^2}{24} (   9\eta^3  - 9 \eta^2 + 11\eta -3 )  \\
=& \frac {R^2}{24} ( \eta - \frac 13) \Big [(  3\eta - 1 )^2 + 8 \Big ] = C_3(\eta_0, c_0)R.
\end{align*}

If $x=0$. Then 
$$\begin{aligned}
RHS \ge& \frac {\eta}{12} (3\eta^2 + 1) R^2  + \frac {\gamma \eta R^2}{\sqrt{3}} \cdot  \Big(  \frac {\eta}2 - \frac 1{2\sqrt{3}} \Big)  \\
=& \frac {\eta R^2}{12} \Big[  3\eta^2 + 1   +2\gamma (\sqrt{3}\eta  - 1)  \Big]  .
\end{aligned}   \eqno(3.2)$$

Note that $\sqrt{3}\eta  - 1<0$, we have
\begin{align*}
3\eta^2 + 1   &+2\gamma (\sqrt{3}\eta  - 1)  \\
>& 3\eta^2 + 1   + 2(1+\sqrt{3}) \cdot (\sqrt{3}\eta  - 1) \\
=& 3\eta^2  + ( 6 + 2\sqrt{3}) \eta  - (1+2\sqrt{3})\\
=& 3 \Big(    \eta    -  \frac   {\sqrt{5+4\sqrt{3}}  -  (1+\sqrt{3})  }  {\sqrt{3}}   \Big) \cdot  \Big(    \eta    +  \frac   {\sqrt{5+4\sqrt{3}}  +  (1+\sqrt{3})  }  {\sqrt{3}}   \Big)  .
\end{align*}   

Thus
$$RHS\ge C_4(\eta_0, c_0)R.$$

Combine the above argument, we have
$$I\ge C_5(c_0, \eta_0)R.$$

So by choosing $\delta=\delta(c_0, \eta_0, \gamma)=\min\{ 1, 2C_2, 2C_5\}$,  the inequality $(3.1)$ holds.  The proof of Lemma 3.1 is complete. 

\end{proof}

\section{A gap theorem of four-dimensional shrinking GRS}
Suppose $(M^4, g)$ is a complete GSS. Then there are a smooth function $f$ and a positive constant $\rho$, such that
$$R_{ij}+ \nabla_i\nabla_j f= \rho g_{ij}.$$

 It is well known that  there exist a self-similar solution of Ricci flow as follow
$$g(t)=\tau(t)\varphi_t^*(g), \ t\in ( -\infty, \frac 1{2\rho}),$$
where $\tau(t)=1-2\rho t$, and $\varphi_t$ is a family of diffeomorphisms.

Now we can prove Theorem 1.1.

\begin{proof}{\bf of Theorem 1.1.}
It is well known that any shrinking GRS with nonnegative Ricci curvature either is flat, or has positive scalar curvature $R\ge r_0>0$ for some positive constant $r_0=r_0(g)$. In the following, we always assume the soliton has positive scalar curvature $R\ge r_0>0$ (cf. \cite{Ni}).

We will argue by contradiction.  Denote by 
$$\eta_0=\sup\limits_{x\in (M^4, g)} \frac{b(x)}{R(x)}\le 1.$$ 
If $\eta_0 \le \tilde{c}$, then we have $\lambda_1 + \lambda_2 \ge \frac {1- \tilde{c}}2 R = c_0 R$, and we have done. 
If not, then $\eta_0 > \tilde{c}$. By the assumptions, we see that the self-similar solution $g(t)_{t\in [0, \frac 1{10\rho}]} $ has nonnegative and uniformly bounded Ricci curvature
with $g(0)=g$.

Then by Lemma 3.1, 
there exist a positive constant $\delta=\delta(r_0, \eta_0, \gamma) \in (0, 1]$, such that $$b \le (\eta_0 -\delta t )R.$$ is preserved under the Ricci flow for all small $t \in [0, T']$, where $T'=\min\{ \frac 1{10\rho}, \frac{ \eta_0  - \tilde{c} }2  \}$.

Hence we have
$$b \le (\eta_0 -\delta T' )R$$
at every point.
But this is impossible. Since there exist some point $p\in M$, such that $b(p) \ge (\eta_0 - \frac{\delta}2 T' )R(p)$ at time $t=0$. Note that
$g(t)$ only changes by scaling and a diffeomorphism on $M^4$, and then exist some point $q\in M$, such that at time $t=T'$, 
$$b(q, T')=\frac 1{1-2\rho T'} b(p) \ge \frac 1{1-2\rho T'} (\eta_0 -\frac{\delta}2 T' )R(p)=(\eta_0 -\frac{\delta}2 T' )R(q, T'),$$
which is contradictive with $b(q, T') \le (\eta_0 -\delta T' )R(q, T')$.

And we complete the proof of Theorem 1.1.

\end{proof}

Next, we follow a similar argument to prove Theorem 1.3.
\begin{proof}{\bf of Theorem 1.3.}
Obviously, we only need to show that 
$$\eta_0=\sup\limits_{x\in (M^4, g)} \frac{b(x)}{R(x)}\le \frac13.$$ 

If not, $\eta_0>\frac 13$, then we can prove the following assertion.

\begin{claim} 
Suppose we have a solution of Ricci flow $g(t)_{t\in [0, T]}$ on a four-manifold with uniformly bounded and nonnegative Ricci curvature, and satisfying the pinched condition $(**)$ 
for all $t \in [0, T]$.

If $R\ge r_0$ and $b \le \eta_0 R$
holds for some positive constant $r_0>0$ and $\eta_0 > \frac 13$ at time $t=0$ . Then there exist a positive constant $\delta=\delta(r_0, \eta_0, \gamma) \in (0, 1]$, such that 
$$b \le (\eta_0 -\delta t )R$$
holds for all $t \in [0, T']$, where $T'=\min\{ T, \frac{ \eta_0  - \frac 13 }2  \}$.
\end{claim}

For the proof of Claim 4.1, we check the argument of Lemma 3.1. Then we only need to get a positive lower bound of (3.2). Indeed,
\begin{align*}
3\eta^2 + 1   +2\gamma (\sqrt{3}\eta  - 1) 
\ge& 3\eta^2 + 1   +2\cdot \frac{1+\sqrt{3}}{\sqrt{3}} \cdot (\sqrt{3}\eta  - 1) \\
=& 3(\eta - \frac 13) \cdot (\eta + \frac{2+\sqrt{3}}{\sqrt{3}} ) \ge C(\eta_0).
\end{align*}

Thus Claim 4.1 holds, which develops a contradiction. And then we obtain Theorem 1.3.

\end{proof}

\bibliographystyle{amsplain}

\end{document}